\documentclass[twoside,11pt]{article}
\usepackage{pgf,tikz}
\usetikzlibrary{arrows}
\usetikzlibrary{calc}
\usepackage[sort&compress]{natbib} 
\usepackage{graphicx, subfigure}
\usepackage[top=1in,bottom=1in,left=1in,right=1in]{geometry}
\usepackage{amsfonts, amsmath, amssymb, amsthm, constants, bbm}
\usepackage{MnSymbol}
\usepackage{todonotes}

\usepackage{color}
\definecolor{darkred}{RGB}{100,0,0}
\definecolor{darkgreen}{RGB}{0,100,0}
\definecolor{darkblue}{RGB}{0,0,150}
\definecolor{lightblue}{RGB}{150,190,220}

\usepackage{hyperref}
\hypersetup{colorlinks=true, linkcolor=darkred, citecolor=darkgreen, urlcolor=darkblue}
\usepackage{url}

\newtheorem{thm}{Theorem}

\newtheorem{lem}{Lemma}

\theoremstyle{remark}
\newtheorem{Def}{Definition}
\newtheorem{rem}{Remark}

\def\beq{\begin{equation}} 
\def\eeq{\end{equation}}
\def\beqn{\begin{eqnarray*}}
\def\eeqn{\end{eqnarray*}}
\def\Bitem{\begin{itemize}\setlength{\itemsep}{.2in}}
\def\bitem{\begin{itemize}\setlength{\itemsep}{.05in}}
\def\eitem{\end{itemize}}
\def\Benum{\begin{enumerate}\setlength{\itemsep}{.2in}}
\def\benum{\begin{enumerate}\setlength{\itemsep}{.05in}}
\def\eenum{\end{enumerate}}
\def\bmult{\begin{multline*}}
\def\emult{\end{multline*}}
\def\bcenter{\begin{center}}
\def\ecenter{\end{center}}
\def\bframe{\begin{frame}}
\def\eframe{\end{frame}}

\newcommand{\thmref}[1]{Theorem~\ref{thm:#1}}

\newcommand{\lemref}[1]{Lemma~\ref{lem:#1}}
\newcommand{\secref}[1]{Section~\ref{sec:#1}}
\newcommand{\figref}[1]{Figure~\ref{fig:#1}}

\newcommand{\remref}[1]{Remark~\ref{rem:#1}}
\newcommand{\defref}[1]{Definition~\ref{def:#1}}
\newcommand{\tabref}[1]{Table~\ref{tab:#1}}

\DeclareMathOperator{\dist}{dist}

\DeclareMathOperator{\acos}{acos}

\def\cC{\mathcal{C}}

\def\cE{\mathcal{E}}

\def\cU{\mathcal{U}}

\def\cX{\mathcal{X}}

\def\bbI{\mathbb{I}}

\def\bbR{\mathbb{R}}

\newcommand{\E}{\operatorname{\mathbb{E}}}
\renewcommand{\P}{\operatorname{\mathbb{P}}}

\def\Bin{\text{Bin}}

\renewcommand{\>}{\rangle}

\def\eps{\varepsilon}

\def\symd{\triangle}
\def\comp{\mathsf{c}}
\def\1{\mathbbm{1}}
\newcommand{\IND}[1]{\bbI\{ #1 \}}

\def\C{C}
\def\F{F}
\def\G{G}
\def\W{W}

\definecolor{purple}{rgb}{0.4,.1,.9}

\begin{document}
\thispagestyle{empty}

\title{Minimax Estimation of the Volume of a Set with Smooth Boundary}
\author{
Ery Arias-Castro\footnotemark[2] \and
Beatriz Pateiro-L\'opez\footnotemark[3] \and
Alberto Rodr\'iguez-Casal\footnotemark[3]
}
\date{}
\renewcommand{\thefootnote}{\fnsymbol{footnote}}
\footnotetext[2]{Department of Mathematics, University of California, San Diego, CA, USA.}
\footnotetext[3]{Departamento de Estat\'istica e Investigaci\'on Operativa, Facultade de Matem\'aticas, Universidade de Santiago de Compostela, Spain.}
\renewcommand{\thefootnote}{\arabic{footnote}}


\maketitle

\begin{abstract}
We consider the problem of estimating the volume of a compact domain in a Euclidean space based on a uniform sample from the domain.  We assume the domain has a boundary with positive reach.  We propose a data splitting approach to correct the bias of the plug-in estimator based on the sample $\alpha$-convex hull.  We show that this simple estimator achieves a minimax lower bound that we derive.  Some numerical experiments corroborate our theoretical findings.

\end{abstract}

\section{Introduction}

We consider the problem of estimating the volume of a compact domain\footnote{For us a compact domain is a bounded subset which coincides with the closure of its interior.} $S$ of a Euclidean space based on an IID sample from the uniform distribution supported on $S$.  Concretely, we are given a set of points $\cX_n=\{X_1,\ldots,X_n\}$, which we assume are drawn independently from the uniform distribution on $S \subset \bbR^d$, and our goal is to estimate the volume of $S$.  Let $\partial S$ denote the boundary of a set $S \subset \bbR^d$, namely $\partial S = \bar{S} \cap \overline{S^\comp}$, where $\bar{S}$ denotes the closure of $S$ and $S^\comp = \bbR^d \setminus S$ is the complement of $S$.  We assume the following:
\beq\label{rolling}
\text{Both $S$ and $S^\comp$ satisfy the $r$-rolling condition.}
\eeq

\begin{Def}
A set $S$ is said to fulfill the $r$-rolling condition if for any
$x\in \partial S$ there is a open ball $B$ with radius $r$ such that $B \cap S=\emptyset$ and $x\in \partial B$.
\end{Def}

Our assumption is equivalent to requiring that $S$ and $S^\comp$ are $r$-convex \citep{MR0077161}.  If, in addition, $S$ is equal to the closure of its interior (which we assume henceforth), then this is also equivalent to asking that $\partial S$ has reach $\ge r$ \citep{MR0110078}.  See \citep{bea-tesis,cuefrapat12,walther97,Walther99}.  Effectively, when $\partial S$ has bounded curvature, the condition is satisfied if $r>0$ is small enough.

\subsection{Lower bound}
Let $\mu$ denote the Lebesgue measure of $\bbR^d$.  Also, let $\E_S$ denote the expectation corresponding to $X_1,\dots,X_n$ sampled IID from the uniform distribution on $S$.
We prove the following lower bound in \secref{lower}.

\begin{thm} \label{thm:lower}
Let $\cC_r(\delta)$ denote the class of the convex sets $S$ satisfying \eqref{rolling} with diameter at most $\delta$.  Assume $\delta > 4r$.  There is a numerical constant $C>0$ such that
\beq
\inf_{\varphi} \sup_{S \in \cC_r(\delta)} \E_S \left[|\varphi(\cX_n)-\mu(S)|\right] \ge C \delta^2 n^{-(d+3)/(2d+2)},
\eeq
where the infimum is over all (measurable) functions $\varphi: \bbR^{dn} \mapsto \bbR$.
\end{thm}


\subsection{Our estimator}
We propose an estimator in \secref{estimate} based on the set estimator of \cite{poincare}.

\begin{Def}
A set $S$ is said to be $\alpha$-convex if for any point $x \notin \bar{S}$ there is a open ball $B$ of radius $\alpha$ such that $x\in B$ and $B \cap \bar{S}=\emptyset$ \citep{MR0077161}.  Given a set $S$, its $\alpha$-convex hull is the smallest $\alpha$-convex set that contains $S$.  It is denoted by $C_\alpha(S)$.
\end{Def}
The notion of $r$-convex hull is closely related to the notion of $\alpha$-shape, well-known in computational geometry \citep{MR713690}.

The set estimator of \cite{poincare} is $C_\alpha(\cX_n)$, the $\alpha$-convex hull of the sample $\cX_n$.  It happens that the plug-in estimator $\mu(C_\alpha(\cX_n))$ is only able to achieve the rate $O(n^{-2/(d+1)})$.
The estimator we propose corrects the bias of this estimator as follows.  The sample is randomly divided into two subsamples. The first subsample is used to estimate $S$ via its $\alpha$-convex hull, denoted $\hat S$, while the second subsample is used to estimate the volume of $S \setminus \hat S$.  Thus the procedure is based on sample splitting.
In detail, randomly split the sample $\cX_n$ into two subsamples $\cX'_n$ and $\cX''_n$ of respective sizes $m$ and $n-m$, where $m$ is a given integer (for example, $m = [n/2]$).
The estimator is computed in several steps: 
\benum 
\item Form $\hat S := C_\alpha(\cX'_n)$, the $\alpha$-convex hull of the first subsample.
\item Compute $\hat p := \frac1{n-m} \# (\cX''_n \setminus \hat S)$.  This is the proportion of points in the second subsample that fall outside the $\alpha$-convex hull of the first subsample.
\item Return the estimator
\beq\label{V-hat}
\hat V = \frac{\mu(\hat S)}{(1-\hat p) \vee 1/2}.
\eeq
\eenum

\begin{thm} \label{thm:simple}
Assume that $S$ satisfies \eqref{rolling}.  Fix $\alpha \in (0,r]$ and $\beta \in (0,1/2)$, and take $m$ such that $\beta \le m/n \le 1-\beta$.  Then estimator $\hat V$ defined in \eqref{V-hat} satisfies
\beq
\E_S\big[|\hat V - \mu(S)|\big] \le c\, n^{-(d+3)/(2d+2)},
\eeq
for a constant $c>0$ not depending on $n$.
\end{thm}

Comparing with \thmref{lower}, we see that our estimator \eqref{V-hat} achieves the minimax rate over the class of sets that satisfy \eqref{rolling} (and are not necessarily convex).
Our method of estimation can be easily enhanced to provide a confidence interval.  We briefly discuss this issue and perform some numerical experiments to evaluate our proposal.

\subsection{Related work}

\cite{MR0169139} consider the estimation of the area of a convex set $S \subset \bbR^2$ with bounded curvature (conditions that imply \eqref{rolling}) using the area of the sample convex hull, obtaining a precise rate of convergence in expectation of order $O(n^{-2/3})$.
\cite{MR1641826} extend their results to other sampling distributions.
Very recently, \cite{baldin2015unbiased} reconsider the case of a uniform sampling distribution, but with the added assumption that the sample size is Poisson distributed --- in which case the sample comes from a Poisson spatial process with constant intensity over the domain of interest.  Under some conditions, they derive the UMVU (uniformly of minimum variance among unbiased estimators) based on a bias correction, but without sample splitting.

\cite{MR1226450} consider the problem of volume estimation in an image model.  One of the settings they assume that $S$ is of the form $S = \{(x,y) \in [0,1]^2 : y \ge g(x)\}$ for some function $g$ with a given H\"older smoothness.  Then the data are of the form $(Z_1, Y_1), \dots, (Z_n, Y_n)$, with $Z_1, \dots, Z_n$ IID uniform in $[0,1]^2$ and $Y_i =  \xi_i \IND{Z_i \in S}$, where the $\xi_i$'s are IID Bernoulli (independent of the $X_i$'s) and represent the noise.  In this setting, they prove a lower bound and provide a rather complex estimator that achieves that lower bound within a poly-logarithmic factor.  The class of H\"older smoothness of order 2 is very close to our setting, and for that class \cite{MR1226450} obtain the same error rate as we do here.
This work is refined and extended by \cite{gayraud1997estimation}, who obtains similar results in arbitrary dimension with unknown sampling distribution.  The case of a convex support set is also covered.  The underlying method uses sample splitting.

The work of \cite{gayraud1997estimation}, complemented by that of \cite{baldin2015unbiased}, shows that the minimax estimation rate under the assumption of convexity (without smoothness assumption) is $n^{-(d+3)/(2d+2)}$, meaning, the same as the minimax estimation rate under the $r$-rolling condition (without convexity assumption).  \thmref{lower} shows, in fact, that adding to the $r$-rolling condition the assumption of convexity does not make the problem substantially easier from a minimax standpoint.

\subsection{Content}

In \secref{lower} we prove the lower bound stated in \thmref{lower}.
In \secref{estimate} we study our estimator and establish \thmref{simple}.
Some numerical experiments are presented in \secref{numerics}.
We discuss some extensions and open problems in \secref{discussion}.

\section{Minimax lower bound}
\label{sec:lower}

In this section we prove \thmref{lower}.
We employ a simple form of Le Cam's method as expounded in \cite[Lem 1]{MR1462963}.
The construction that follows is similar to that of \cite{MR1332579} for the problem of set estimation.

Consider the ball centered at the origin of radius $r_0 := \delta/2 > 2 r$, denoted $B_0$.
Let $y_1, \dots, y_m$ denote a $2\eps$-packing of $\partial B_0$ of maximal size, so that $m \asymp \eps^{-(d-1)}$, as is well-known.
The intersection of $\partial B(y_j, \eps)$ and $\partial B_0$ is the sphere $\partial B(y_j, \eps) \cap H_j$, where $H_j := \{x \in \bbR^d : \<x, y_j\> = r_0^2-\eps^2/2\}$ is a hyperplane.
Let $\theta = \acos(1 -\eps^2/2r_0^2)$, so that $2\theta$ is the aperture of the cone with apex the origin and with base $H_j \cap B_0$.  Let $C_j$ denote the corresponding infinite cone.  Define another hyperplane $K_j = \{x \in \bbR^d : \<x, y_j\> = r_0^2 - r_0 h\}$, where $h := (r_0 -r) (1 -\cos \theta) = (r_0 -r) \eps^2/2r_0^2$.  Note that $H_j$ is at distance $r_0 - \eps^2/2r_0$ from the origin, while $K_j$ is parallel to $H_j$ and at distance $r_0 - h$ from the origin.
Define the half-space $\bar K_j = \{x \in \bbR^d : \<x, y_j\> \le r_0^2 - r_0 h\}$ and then $\bar H_j$ analogously.
Let $Q_j$ denote the points $x \in B_0$ with the property that there is a ball $B$ of radius $r$ such that $x \in B \subset B_0 \cap \bar K_j$.  In other words, we remove from $B_0$ the cap defined by $K_j$ and obtain $Q_j$ by rolling a ball of radius $r$ inside the resulting set.
By construction, $B_0 \cap \bar H_j \subset Q_j \subset B_0 \cap \bar K_j$, and in particular the different sets $B_0 \cap Q_j^\comp$, as $j$ varies, do not intersect.  (The latter is because $\|y_j - y_{j'}\| > 2\eps$ when $j \ne j'$.)
See \figref{ball-cut} for an illustration in dimension $d=2$.
\begin{figure}[ht]
\centering
\def\ro{3}
\def\r{1.2}
\def\eps{2.5}
\def\h{0.625} 
\def\dh{1.95833}  
\def\dk{2.375} 
\def\th{49.24864}

\def\cx{{sqrt(2*(\ro-\r)*\h-\h*\h)}}
\def\cy{{-(\ro-\h-\r)}}
\def\au{{sqrt(\ro*\ro-\dh*\dh)}}

\begin{tikzpicture}
\colorlet{lightgray}{black!25}
\draw[densely dashed,fill=lightgray](0,0) circle (\ro);
\draw[white,fill=white](-\ro,-\dh)--(-\ro,-\ro)--(\ro,-\ro)--(\ro,-\dh);
\fill[lightgray] (\cx,\cy) circle (\r);
\fill[lightgray] (-\cx,\cy) circle (\r);
\draw[lightgray,fill=lightgray](-\cx,0)--(-\cx,-\dk)--(\cx,-\dk)--(\cx,0);


\draw[densely dashed](0,0) circle (\ro); 
\draw [densely dashed, domain=-3:3] plot(\x,{(-\dh-0*\x)/1}) node[right,above=-1.5mm] {\scriptsize  $H_j$}; 
\draw [densely dashed, domain=-3:3] plot(\x,{(-\dk-0*\x)/1}) node[right,above=-1.5mm] {\scriptsize  $K_j$}; 

\fill (0,-\ro) circle (1pt) node[below] {\scriptsize $y_j$};

\draw[<->,color=black] (-\ro,0) -- (0,0) node[midway,above] {\scriptsize  $r_0$};
\draw[<->,color=black] (0,-\dk) -- (0,0) node[midway,left] {\scriptsize  $r_0-h$};
\draw[densely dashed,color=black] (\au,-\dh) -- (0,0) ;
\draw[densely dashed,color=black] (\cx,\cy) circle (\r);
\draw[<->,color=black] (\cx,\cy) -- (\au,-\dh)  node[midway,above] {\scriptsize  $r$};
\draw[densely dashed](0,0) +(270:0.3cm) arc (270:{\th+270}:0.3cm);
\node at (0.15,-0.4) {\scriptsize  $\theta$};

\draw[thick](0,0) +({270-\th}:\ro) arc ({270-\th}:{-90+\th}:\ro);
\draw[thick](\cx,\cy) +(270:\r) arc (270:{270+\th}:\r);
\draw[thick](-\cx,\cy) +(270:\r) arc (270:{270-\th}:\r);
\draw[thick](-\cx,-\dk)--(\cx,-\dk);
\end{tikzpicture}
\caption{
The ball $B_0 = B(0, r_0)$ is smoothly `dented' to obtain $Q_j$, represented in gray.
}
\label{fig:ball-cut}
\end{figure}
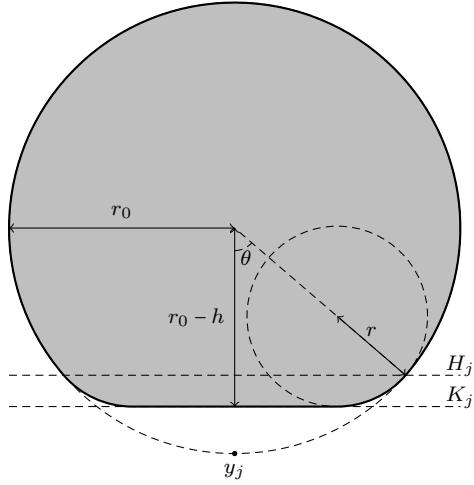

For $\omega = (\omega_1, \dots, \omega_m) \in \{0, 1\}^m$, let
\beq
S_\omega = B_0 \cap \bigcap_{\{j : \omega_j = 1\}} Q_j.
\eeq
By construction, for any $\omega$, both $S_\omega^\comp$ and $S_\omega$ satisfy the $r$-rolling condition, the latter being convex by \lemref{convex} in the Appendix.

Let $\Pi_\ell$ denote the uniform distribution on $\Omega := \{\omega: |\omega|_1 = \ell\}$, where for $\omega = (\omega_1, \dots, \omega_m) \in \{0, 1\}^m$, we let $|\omega|_1 = \sum_j \omega_j$.
The parameter $\ell$ will be chosen later on.
Define $\eta = \mu(B_0) - \mu(Q_j) = \mu(B_0 \setminus Q_j)$.
By \cite[Lem 1]{MR1462963},
\beq\label{LeCam}
\inf_{\varphi} \sup_{S \in \cC_r(\delta)} \E_S |\varphi(\cX_n) - \mu(S)| \geq \tfrac12 \ell\eta\, \Big(1 - \tfrac12 {\rm TV}(P_0^{\otimes n}, P_1^{\otimes n})\Big),
\eeq
where $P_0$ is the uniform distribution on $B_0$, $P_1$ is the mixture of $\P_{S_\omega}$ when $\omega \sim \Pi_\ell$, and ${\rm TV}$ denotes the total variation metric for distributions.  This is the bound we work with.

We first bound $\eta$, from below but also from above, as this will be needed later on.
Let $\gamma$ denote the angle associated to $K_j$ as $\theta$ is associated to $H_j$, and note that $\gamma = \acos(1 - h^2/r_0^2)$.
We will take $\eps$ small, and as $\eps \to 0$ we have that $m \to \infty$, $\theta \sim \eps/r_0$, and $\gamma \sim \sqrt{2 h/r_0}$.

The volume of a cap of the unit ball in $\bbR^d$ at distance $1-t$ from the origin is equal to
\beq
\frac{\pi^{\frac{d-1}{2}}}{\Gamma\left(\frac{d+1}{2}\right)}\int_0^{\arccos(1-t)}\sin^d(x){\rm d}x \asymp t^{(d+1)/2}, \quad t \to 0.
\eeq
Using this, as $\eps \to 0$,
\beq\begin{split} \label{eta}
\eta &\le \mu(B_0 \setminus \bar H_1) \asymp \eps^{d+1}, \\
\eta &\ge \mu(B_0 \setminus \bar K_1) \asymp h^{(d+1)/2} \asymp \eps^{d+1}, \quad \text{since } h \asymp \eps^2.
\end{split}\eeq

Define
\beq
Z = \frac{{\rm d} P_1^{\otimes n}(\cX_n)}{{\rm d} P_0^{\otimes n}(\cX_n)} = (1 - \ell \eta/\zeta_d r_0^d)^{-n} \ \frac{1}{|\Omega|} \sum_{\omega \in \Omega} \IND{\cX_n \subset S_\omega},
\eeq
where $\zeta_d$ is the volume of the unit ball in dimension $d$.
Then
\beq\label{TV}
{\rm TV}(P_0^{\otimes n}, P_1^{\otimes n}) = \E_0[|Z -1|] \le\sqrt{\E_0(Z^2) -1},
\eeq
where the inequality is Cauchy-Schwarz's.
We have
\beq
\E_0(Z^2) = (1 - \ell \eta/\zeta_d r_0^d)^{-2n} \ \frac{1}{|\Omega|^2} \sum_{\omega, \omega' \in \Omega} \E_0(\IND{\cX_n \subset S_{\omega}} \IND{\cX_n \subset S_{\omega'}}),
\eeq
with
\beq
\E_0(\IND{\cX_n \subset S_{\omega}} \IND{\cX_n \subset S_{\omega'}})
= \E_0(\IND{\cX_n \subset S_{\omega} \cap S_{\omega'}})
= \big(1 -(2\ell - |\omega \wedge \omega'|_1) \eta/\zeta_d r_0^d\big)^{n},
\eeq
where $\omega \wedge \omega' = (\omega_1 \wedge \omega'_1, \dots, \omega_m \wedge \omega'_m)$ when $\omega = (\omega_1, \dots, \omega_m)$ and $\omega' = (\omega'_1, \dots, \omega'_m)$.
Noting that $|\omega \wedge \omega'|_1$ has the hypergeometric distribution with parameters $(m, \ell, \ell)$ when $w,w'$ are IID with distribution $\Pi_\ell$, and letting $V$ denote a random variable with that distribution, we have
\beq\begin{split}
\E_0(Z^2)
&= (1 - \ell \eta/\zeta_d r_0^d)^{-2n} \E\Big[\big(1 -(2\ell - V) \eta/\zeta_d r_0^d\big)^n\Big] \\
&= \E \bigg[\bigg(\frac{1 -(2\ell - V) \eta/\zeta_d r_0^d}{(1 - \ell \eta/\zeta_d r_0^d)^2}\bigg)^n\bigg] \\
&\le \E \bigg[\Big(1 + V \eta/\zeta_d r_0^d + 10 (\ell \eta/\zeta_d r_0^d)^2\Big)^n\bigg] \\
&\le \exp(10 (\ell \eta/\zeta_d r_0^d)^2 n) \E \big[\exp(n V \eta/\zeta_d r_0^d)\big],
\end{split}
\eeq
where in the third line we assumed that $\ell \eta/\zeta_d r_0^d \le 1/2$.
The function $x \mapsto e^{a x}$ (with $a > 0$ fixed) being convex, we may apply~\cite[Th 4]{hoeffding} to bound the last expectation by
\beq
\E \big[\exp(n W \eta/\zeta_d r_0^d)\big],
\eeq
where $W$ is binomial with parameters $(\ell, \ell/m)$.
We then continue
\beq
\E \big[\exp(n W \eta/\zeta_d r_0^d)\big]
= \Big(1 - \tfrac{\ell}{m} + \tfrac{\ell}{m} e^{n \eta/\zeta_d r_0^d}\Big)^\ell
\le \exp\Big[\tfrac{\ell^2}{m} e^{n \eta/\zeta_d r_0^d}\Big].
\eeq
Therefore, we conclude that $\E_0(Z^2) \le 2$ when
\beq\label{cond1}
\exp(10 (\ell \eta/\zeta_d r_0^d)^2 n) \le \sqrt{2},
\eeq
(which implies $\ell \eta/\zeta_d r_0^d \le 1/2$) and when
\beq\label{cond2}
\exp\big[\tfrac{\ell^2}{m} e^{n \eta/\zeta_d r_0^d}\big] \le \sqrt{2}.
\eeq
From \eqref{eta}, we know there is a constant $c_0$ such that $\eta/\zeta_d r_0^d \le c_0 \eps^{d+1}$.  Hence \eqref{cond1} and \eqref{cond2} are implied by
\beq\begin{split}
\ell^2 \le \tfrac1{10 c_0^2} \log(\sqrt{2})/n\eps^{2d+2}, \quad
\ell^2 \le \eps^{1-d} e^{-c_0 n \eps^{d+1}} \log(\sqrt{2}).
\end{split}\eeq
Taking $\eps = n^{-1/(d+1)}$, we can see that we may set $\ell = [c n^{(d-1)/(2d+2)}]$ with $c > 0$ a sufficiently small constant.
Note that $\eta \asymp 1/n$ with this choice of $\eps$ by \eqref{eta}.
This guarantees that, $n$ being large enough, $\E_0(Z^2) \le 2$, and when this is the case, from \eqref{TV}, the RHS of \eqref{LeCam} is lower-bounded by $\eta\ell/4 \asymp (1/n) n^{(d-1)/(2d+2)} = n^{-(d+3)/(2d+2)}$, which concludes the proof of \thmref{lower}.

\section{Performance analysis for our estimator}
\label{sec:estimate}

In this section we analyze our estimator $\hat{V}$ defined in \eqref{V-hat} and prove \thmref{simple}.
We start with \thmref{cr_d}, which bounds $\E\big[\mu(C_\alpha(\cX_n) \symd S)\big]$.  ($\symd$ denotes the symmetric difference).
\thmref{cr_d} generalizes Theorem~1 in \citep{patrod13} to the $d$-dimensional Euclidean space. Although some arguments in the proof of the result (see the Appendix) are analogous to those used in the bidimensional case, the proof of  \thmref{cr_d} is not just an extension of the existing proof in that paper.  In particular, see Lemmas~\ref{lem:inev.Rd1} and~\ref{lem:inev.Rd2}.  The proof is based on unpublished work in \citep{bea-tesis}.

\begin{thm}\label{thm:cr_d}
Let $S$ be a nonempty compact subset of $\mathbb{R}^d$ such that both $S$ and $S^{c}$ satisfy the $r$-rolling condition. Let $X$ be a random variable with probability distribution $P_X$ and support $S$. We assume that the probability distribution $P_X$ satisfies that there exists $\delta>0$ such that $P_X(C)\geq\delta\mu(C\cap S)$ for all Borel subsets $C\subset\mathbb{R}^d$. Let $\mathcal{X}_n=\{X_1,\ldots,X_n\}$ be a random sample from $X$ and let $\{\alpha_n\}$ be a sequence of positive numbers which do not depend on the sample and such that $\alpha_n\leq r$. If the sequence $\{\alpha_n\}$ satisfies
\begin{equation}\label{conv.rn}
\lim_{n\rightarrow\infty}\frac{n\alpha_n^d}{\log n}\rightarrow\infty,
\end{equation}
then
\begin{equation}\label{cotaEE}
\E\big[\mu(C_{\alpha_n}(\cX_n) \symd S)\big]=O\left(\alpha_n^{-(d-1)/(d+1)}n^{-2/(d+1)}\right).
\end{equation}
\end{thm}

We use \thmref{cr_d} and other results in \citep{poincare} to establish \lemref{hull} below.
\thmref{cr_d} is in fact a bit more than what we need, and is really only used to obtain the bound \eqref{hull2} --- by taking $P_X$ to be the uniform distribution on $S$ and $\alpha_n=\alpha\in (0,r]$.

\begin{lem} \label{lem:hull}
Let $S \subset \bbR^d$ be compact and satisfy \eqref{rolling}.
Also, let $\cX_n$ denote a sample of size $n$ from the uniform distribution on $S$.
For any $\alpha \in (0, r]$, there is a constant $c > 0$ not depending on $n$ such that
\beq\label{hull1}
\P(\mu(C_\alpha(\cX_n) \symd S) > \eps) \le c\, \eps^{-d/2} \exp(-n \eps^{(d+1)/2}/c), \quad \forall \eps > 0,
\eeq
and also,
\beq\label{hull2}
\E\big[\mu(C_\alpha(\cX_n) \symd S)\big] \le c n^{-2/(d+1)}.
\eeq
\end{lem}

\begin{rem}
\label{rem:plug-in}
Assuming that $r$ is known and that we choose $\alpha$ accordingly (for example, $\alpha=r$), we have $C_\alpha(\cX_n) \subset S$, which implies
\beq
0 \le \mu(S) - \mu(C_\alpha(\cX_n)) \le \mu(C_\alpha(\cX_n) \symd S),
\eeq
so that, by \eqref{hull2}, the plug-in estimator $\mu(C_\alpha(\cX_n))$ achieves the error rate $O(n^{-2/(d+1)})$.
We conjecture that this is sharp, and if so, the plug-in estimator  does not achieves the error rate obtained in \thmref{lower}, not even within a poly-logarithmic factor.
\end{rem}

\begin{proof}[Proof of \thmref{simple}]
Recall the process leading to $\hat{V}$ in \eqref{V-hat}.
Let $\tilde p = \mu(S \setminus \hat S)/\mu(S)$ and note that $\mu(S) = \mu(\hat S)/(1-\tilde p)$.
Let $V = \mu(S)$, which is what we want to estimate, and let $\hat V_0 = \mu(\hat S)$, which is the plug-in estimate.  Note that $V = \hat V_0/(1-\tilde p)$, so that
\beq
\tilde p = (V - \hat V_0)/V \le \mu(\hat S \symd S).
\eeq
We have
\beq\begin{split}
\E_S |\hat V - V|
&= \E_S\big[|\hat V - V| \, \IND{\tilde p \ge 1/4}\big] \\
& \quad + \E_S \big[|\hat V - V| \, \IND{\tilde p < 1/4, \hat p > 1/2}\big] \\
& \qquad + \E_S \big[|\hat V - V| \, \IND{\tilde p < 1/4, \hat p \le 1/2}\big] .
\end{split}\eeq


Because $\hat V \le 2 \hat V_0 \le 2 V$,
\beq\begin{split}
\E_S\big[|\hat V - V| \, \IND{\tilde p \ge 1/4}\big]
&\le V \P\big(\tilde p \ge 1/4\big) \\
&\le V \P\big(\mu(\hat S \symd S) \ge V/4\big) \\
&\le c_1 e^{-n/c_1},
\end{split}\eeq
for some constant $c_1 > 0$, by \lemref{hull} and the fact that $m \ge \beta n$.

Similarly, we have
\beq\begin{split}
\E_S \big[|\hat V - V| \, \IND{\tilde p < 1/4, \hat p > 1/2}\big]
&\le V \P\big(\tilde p < 1/4, \hat p > 1/2\big) \\
&\le V \P\big(\hat p > 1/2 ~\big|~ \tilde p < 1/4 \big) \\
&\le c_2 e^{-n/c_2},
\end{split}\eeq
for a constant $c_2 >0$, using the fact that, given $\tilde p$, $(n-m) \hat p  \sim \Bin(n-m, \tilde p)$, and applying Bernstein's inequality together with the fact that $n-m \ge \beta n$.

Finally, using the fact that
\beq
|\hat V - V| = \hat V_0 \frac{|\hat p - \tilde p|}{(1-\hat p)(1-\tilde p)} \le \tfrac83 V |\hat p - \tilde p|,
\eeq
when $\tilde p < 1/4$ and $\hat p \le 1/2$, we have
\beq\begin{split}
\E_S \big[|\hat V - V| \, \IND{\tilde p < 1/4, \hat p \le 1/2}\big]
&\le \tfrac83 V \E_S\big[|\hat p - \tilde p|\big] \\
&\le \tfrac83 V \sqrt{\E_S\big[(\hat p - \tilde p)^2\big]} \\
&= \tfrac83 V \sqrt{\E_S\big[\tilde p(1 - \tilde p)/(n/2)\big]} \\
&\le c_3 \sqrt{(1/n)^{2/(d+1)}/n} = c_3 n^{-(d+3)/(2d+2)},
\end{split}\eeq%
for a constant $c_3 >0$, using the Cauchy-Schwarz Inequality, the fact that, given $\tilde p$, $(n-m) \hat p  \sim \Bin(n-m, \tilde p)$, and \lemref{hull}.

Combining all bounds, and noticing that $(d+3)/(2d+2) \le 1$ for all $d \ge 1$, proves the result.
\end{proof}

\begin{rem} \label{rem:simple}
This estimator relies heavily on the fact that the sampling distribution is uniform.  If this is not the case, it can be biased downward or upward.  For example, suppose that $S$ is the unit disc in dimension 2 and that the sampling distribution has the following density
\beq
f(x) = \frac{a \IND{\|x\| \le 1/2} + b \IND{1/2 < \|x\| \le 1}}{\tfrac\pi4 a + \tfrac{3\pi}4 b},
\eeq
where $a,b > 0$.  In that case, with high probability as $n$ becomes large,
\beq
\tilde p = c \mu(S \setminus C_\alpha(\cX'_n))/\mu(S), \quad c:= \frac{b}{a/4 + 3b/4},
\eeq
and by varying $a$ and $b$, $c$ can be any real in $(0,4/3)$.  If $c < 1$, the estimator remains biased downward, while if $c > 1$, it is biased upward, achieving the same rate as the plug-in estimator of \remref{plug-in}.
\end{rem}

\paragraph{Confidence intervals}
Our procedure lends itself naturally to the computation of confidence intervals.  Indeed, we can see that it boils down to computing a confidence interval for $\tilde p$ based on the second half of the sample.  A natural interval is based on $\hat p$, which given $\tilde p$ satisfies $(n-m) \hat p  \sim \Bin(n-m, \tilde p)$.

\section{Numerical experiments}
\label{sec:numerics}

Here we discuss the results of a simulation study that illustrates the performance of the proposed volume estimator $\hat{V}$ of \eqref{V-hat}.
We consider the set
\beq\label{S}
S=\{x\in\mathbb{R}^2:\ 0.25\leq \|x\|\leq 1\}.
\eeq
Note that $S$ is $r$-convex for $r=0.25$ and $\mu(S)=\pi(1-0.25^2)$.

In the first experiment, we generate a sample of size $n$ from the uniform distribution on $S$ and calculate the estimator $\hat{V}$. We consider different sizes $m$ for the subsample $\cX'_n$ (different ways of splitting the sample) and different values of $\alpha$.
Each setting is repeated $B=500$ times.
\figref{eb} shows the mean values $|\mu(S) - \hat V|/\mu(S)$ over the $B$ repeats (error bars represent one standard deviation).  We do the same for $\mu(C_\alpha(\cX_n))$ instead of $\hat V$.

\begin{figure}[!ht]
\centering
\includegraphics[width=0.32\textwidth]{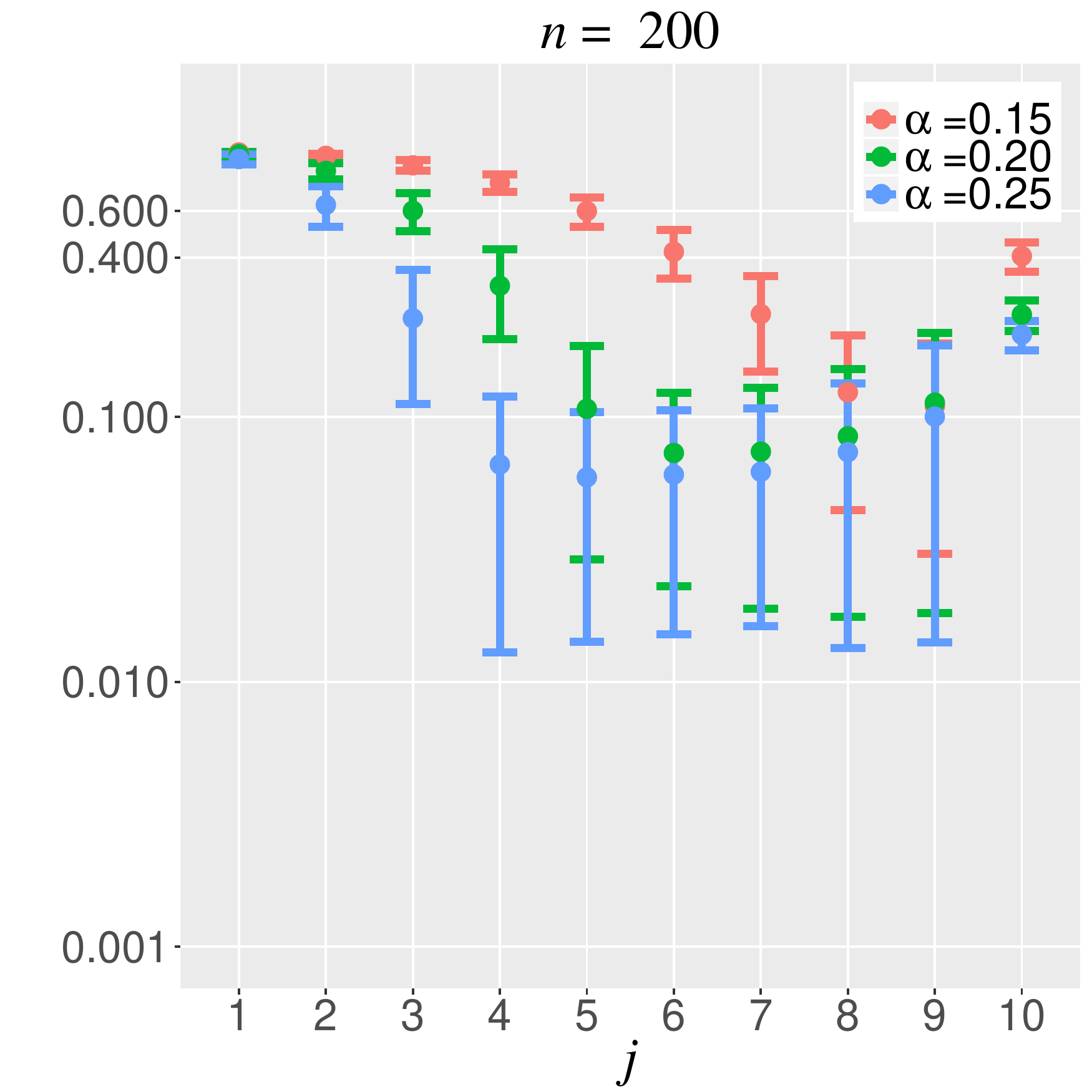}
\includegraphics[width=0.32\textwidth]{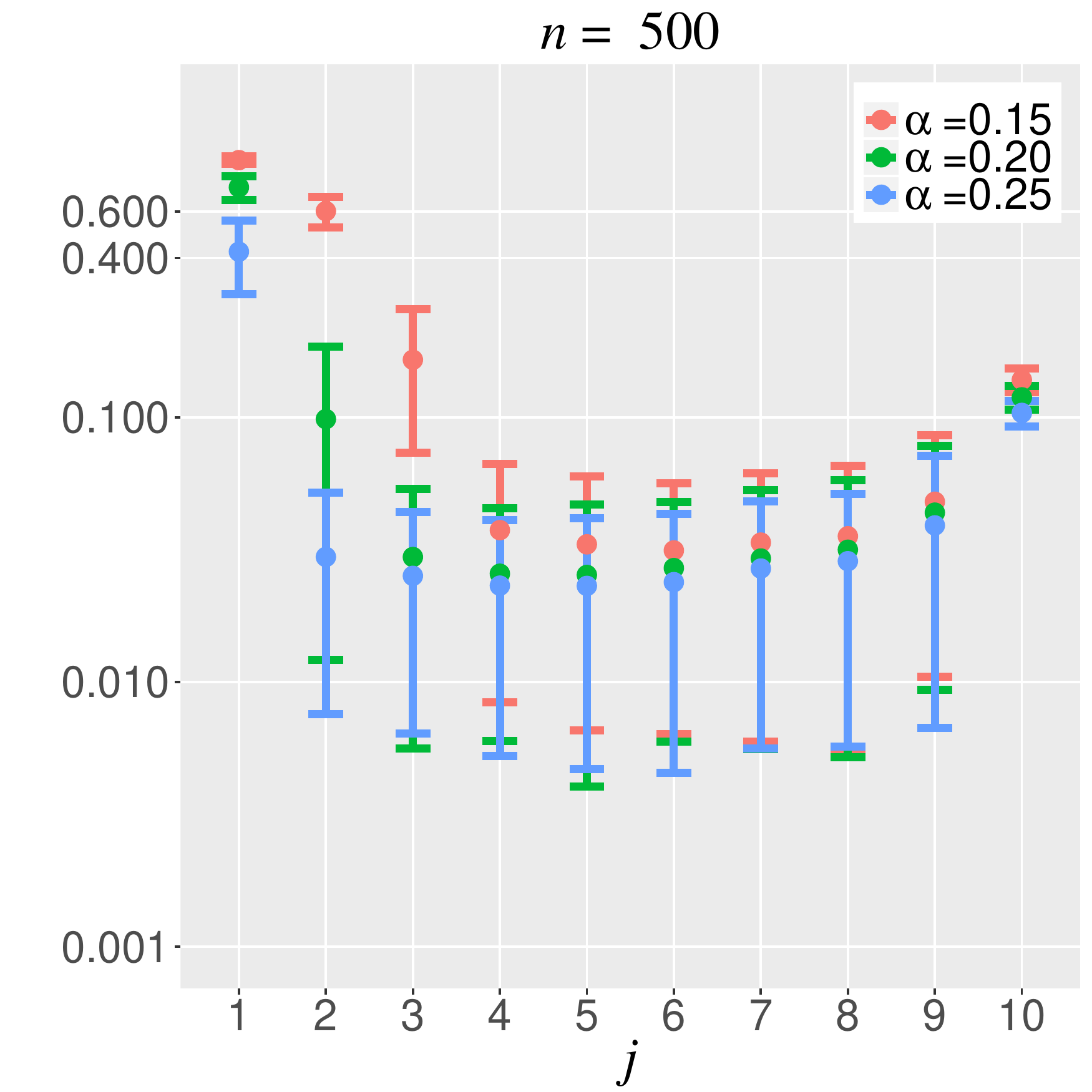}
\includegraphics[width=0.32\textwidth]{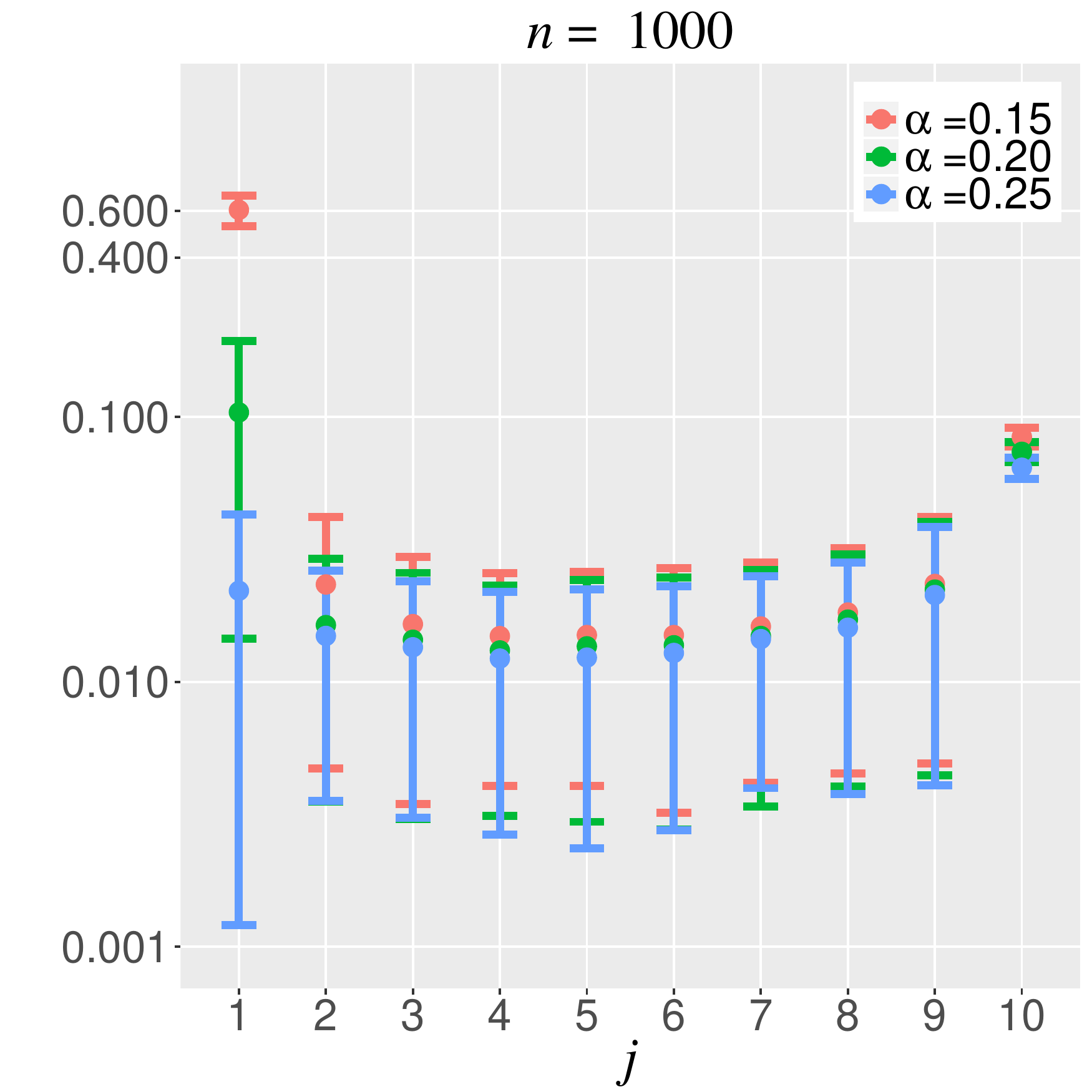}
\caption{Mean values of $|\mu(S) - \hat V|/\mu(S)$ over $B=500$ repeats and error bars representing one standard deviation (a base-10 log scale is used for the $Y$-axis).  In the computation of $\hat V$ we have considered different sizes $m=[nj/10]$ for the subsample $\cX'_n$. The case $j=10$ corresponds to the plug-in estimator.}
\label{fig:eb}
\end{figure}

\paragraph{Bagging}
In the second experiment, we study the same scenarios as in the first experiment, and examine the strategy of performing the sample splitting multiple times.  The new estimator, denoted $\hat V_{\rm bag}$, is obtained by computing ${\hat V}$ for $b=100$ random sample splittings and averaging these.  This is a form of bagging, therefore the name.
\figref{eb_bagg} shows the mean values of $|\mu(S) - \hat V_{\rm bag}|/\mu(S)$ over the $B$ repeats.  ($B = 500$ as before.) As it can be seen, compared to \figref{eb}, this bagging technique reduces the variance of the error.

\begin{figure}[!ht]
\centering
\includegraphics[width=0.32\textwidth]{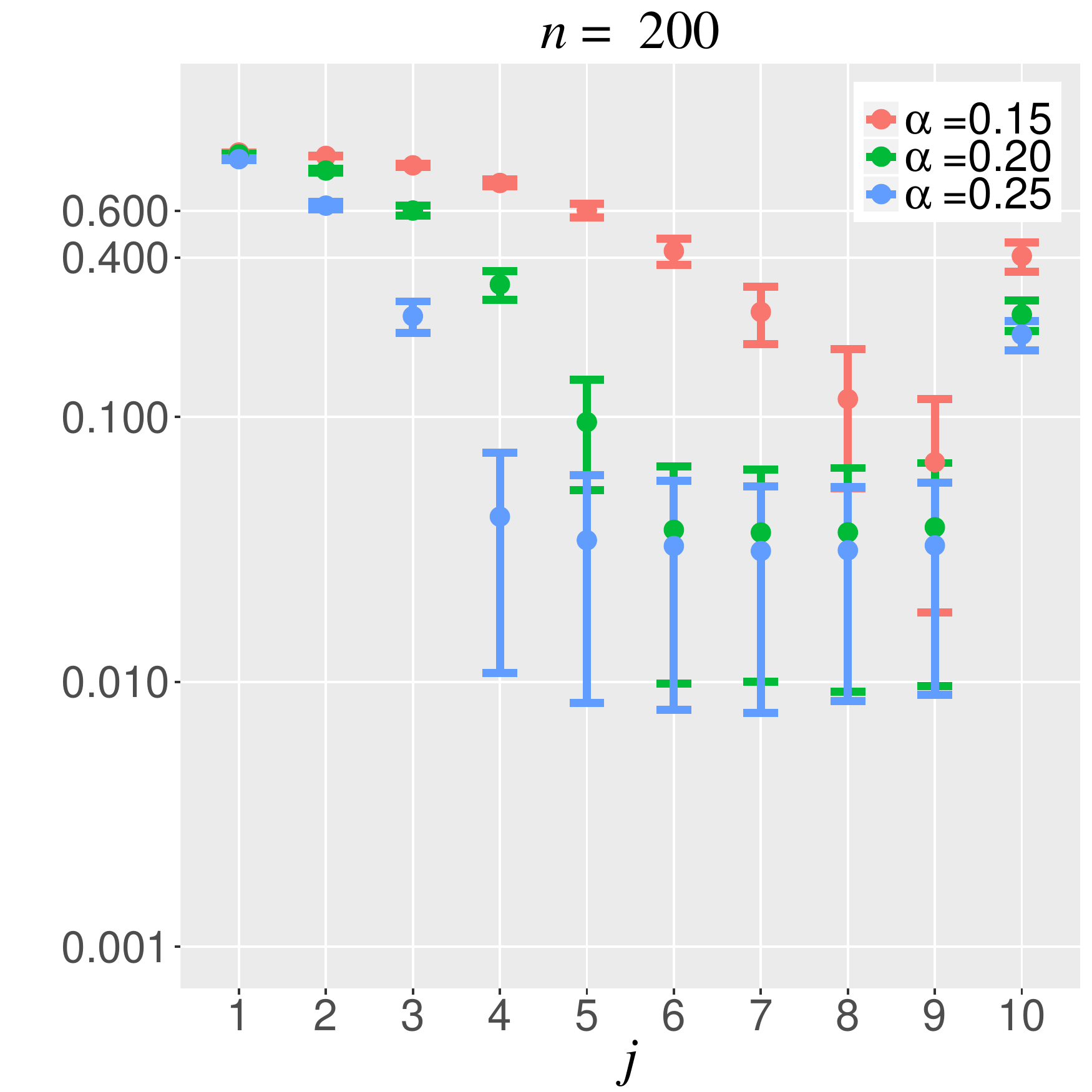}
\includegraphics[width=0.32\textwidth]{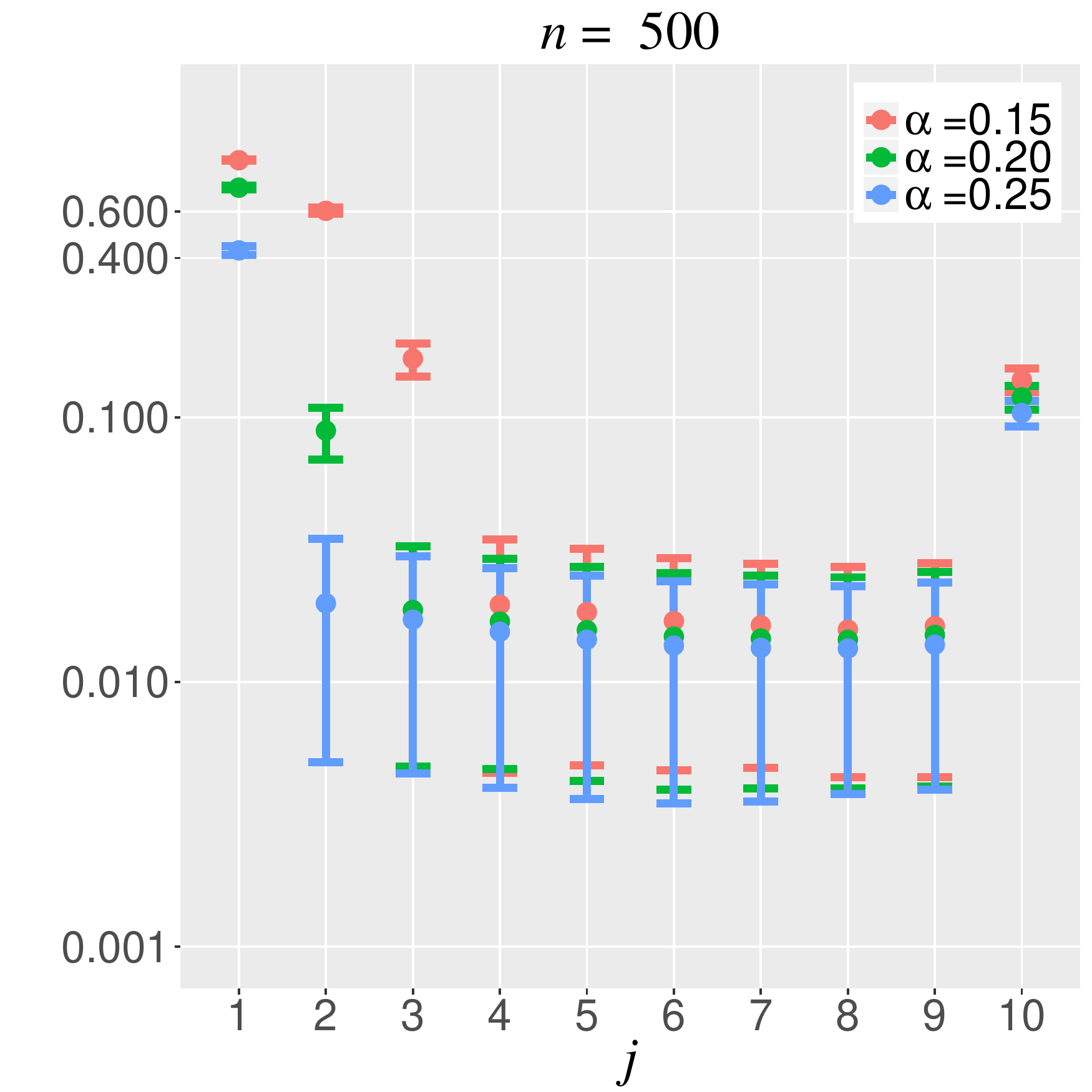}
\includegraphics[width=0.32\textwidth]{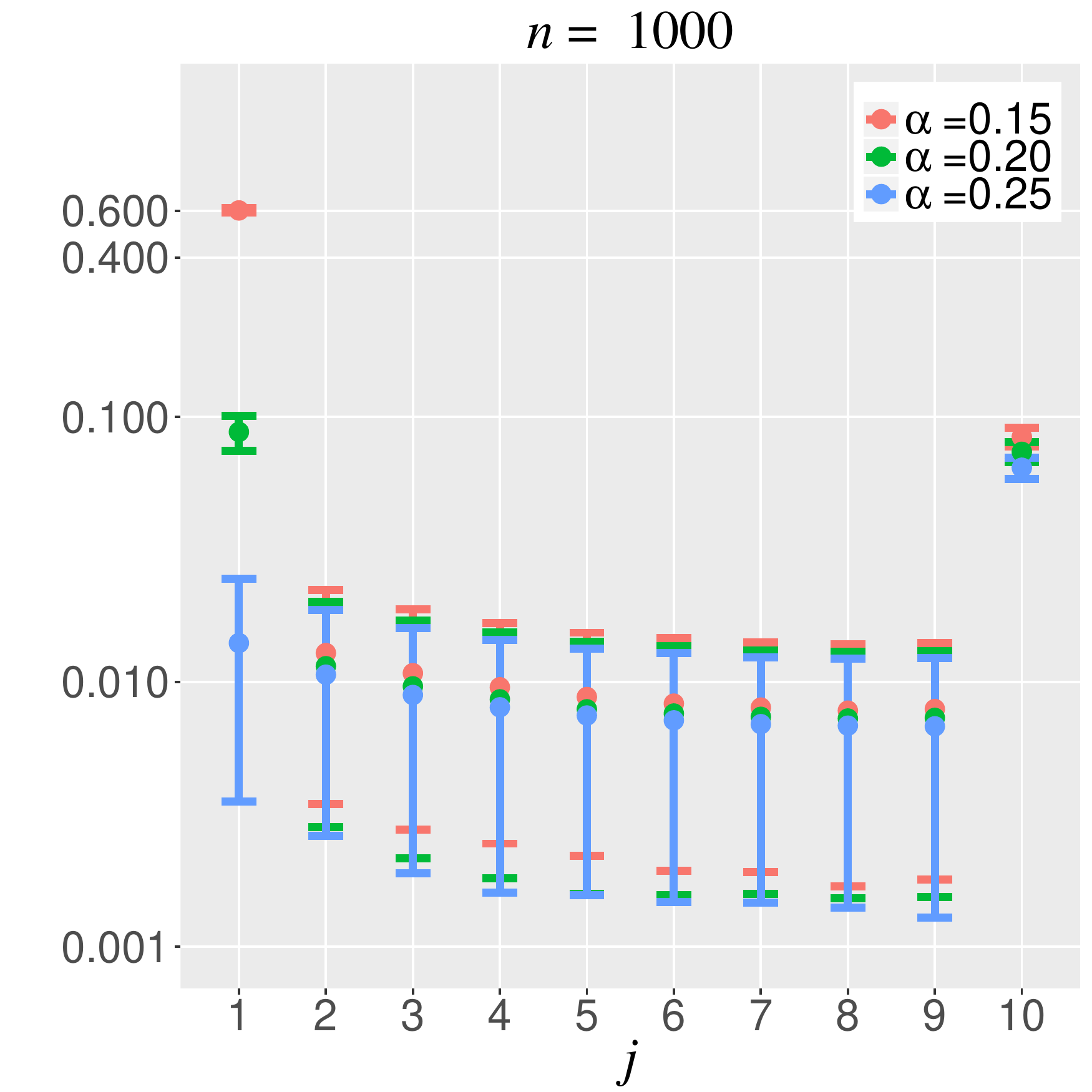}
\caption{Mean values of $|\mu(S) - \hat V_{\rm bag}|/\mu(S)$ over $B=500$ repeats and error bars representing one standard deviation (a base-10 log scale is used for the $Y$-axis).  The bagging was over $b=100$ sample splittings. In the computation of $\hat V$ we have considered different sizes $m=[nj/10]$ for the subsample $\cX'_n$. The case $j=10$ corresponds to $\hat V=\mu(C_\alpha(\cX_n))$.}
\label{fig:eb_bagg}
\end{figure}

\paragraph{Confidence intervals}
As mentioned before, the proposed method lends itself naturally to the computation of confidence intervals for $\mu(S)$ based on the computation of confidence intervals for $\tilde p$.  We use the method of \cite{wilson1927probable} for that purpose.
Results of the estimated coverage probability and estimated mean length  of the confidence intervals for different nominal confidence levels are shown in \tabref{ciW}.  (This is just meant as a proof of concept since there are no other methods we know off to compare this with.)

\begin{table}[!ht]
\centering\small
\begin{tabular}{llrrrrrrrrrr}
  \hline
&Level& 0.50 & 0.55 & 0.60 & 0.65 & 0.70 & 0.75 & 0.80 & 0.85 & 0.90 & 0.95 \\
\hline
$n=200$& Coverage & 0.52 & 0.55 & 0.61 & 0.65 & 0.70 & 0.74 & 0.80 & 0.85 & 0.90 & 0.94 \\
  &Length & 0.31 & 0.35 & 0.39 & 0.43 & 0.48 & 0.53 & 0.59 & 0.67 & 0.76 & 0.91 \\

$n=500$& Coverage & 0.50 & 0.55 & 0.60 & 0.64 & 0.69 & 0.74 & 0.81 & 0.87 & 0.91 & 0.95 \\
  &Length & 0.11 & 0.13 & 0.14 & 0.16 & 0.18 & 0.20 & 0.22 & 0.25 & 0.28 & 0.34 \\

 $n=1000$& Coverage & 0.50 & 0.55 & 0.58 & 0.64 & 0.68 & 0.73 & 0.77 & 0.82 & 0.89 & 0.94 \\
  &Length & 0.06 & 0.07 & 0.08 & 0.08 & 0.09 & 0.10 & 0.12 & 0.13 & 0.15 & 0.18 \\

   \hline

\end{tabular}
\caption{Coverage and length of the confidence interval for $\mu(S)$  based on a confidence interval for $\tilde p$.  We split the sample in half (meaning, we used $m=n/2$) and used $\alpha=0.25$.
Each setting is repeated $B=500$ times and what are shown are the averages of the $B$ repeats.}
\label{tab:ciW}
\end{table}

\paragraph{The convex case}
We replicated the study in \citep{baldin2015unbiased} to compare the performance of our estimator $\hat V$ with  that of the estimators discussed in that paper for the convex case. Data points are simulated for an ellipse, $S$, with center at the origin, major axis of length 10 and minor axis of length 4; see \figref{comp} (left). More specifically, for different values of $n$, we generated $B=500$ samples from a Poisson spatial process over $S$ with constant intensity $\lambda=n/\mu(S)$. The size of each sample, $N$, is Poisson distributed with mean $n$. For the computation of $\hat V$ we randomly split each sample into two subsamples of equal size and compute the $\alpha$-convex hull of the first subsample with $\alpha=10$.  We base our choice of $\alpha$ on the fact that, under the assumption of convexity, the  $\alpha$-convex hull estimator works reasonably well for large values of $\alpha$. \figref{comp} (right) shows, for the considered estimators, the RMSE normalized by the true area based on the $B = 500$ Monte Carlo iterations for each $n$. We use the same notation as in \citep{baldin2015unbiased}.
Our estimator $\hat{V}$ performs slightly better than the rate-optimal estimator based on sample splitting by \cite{gayraud1997estimation}, denoted by $\hat{\upsilon}_G$.  The best performance corresponds to $\hat{\upsilon}_{oracle}$, altough its computation depends on the unknown intensity $\lambda$. The estimators $\hat{\upsilon}_{plugin}$ and $\hat{\upsilon}$, for the case of unknown intensity $\lambda$, also perform well. As already pointed out by \cite{baldin2015unbiased}, all these methods clearly outperform the results of other estimators that are not rate-optimal, such as the Lebesgue measure of the convex hull of the sample, denoted by $|\hat{C}|$, and the so-called naive oracle estimator $N/\lambda$.

\begin{figure}[!ht]
\centering
\includegraphics[width=0.4\textwidth]{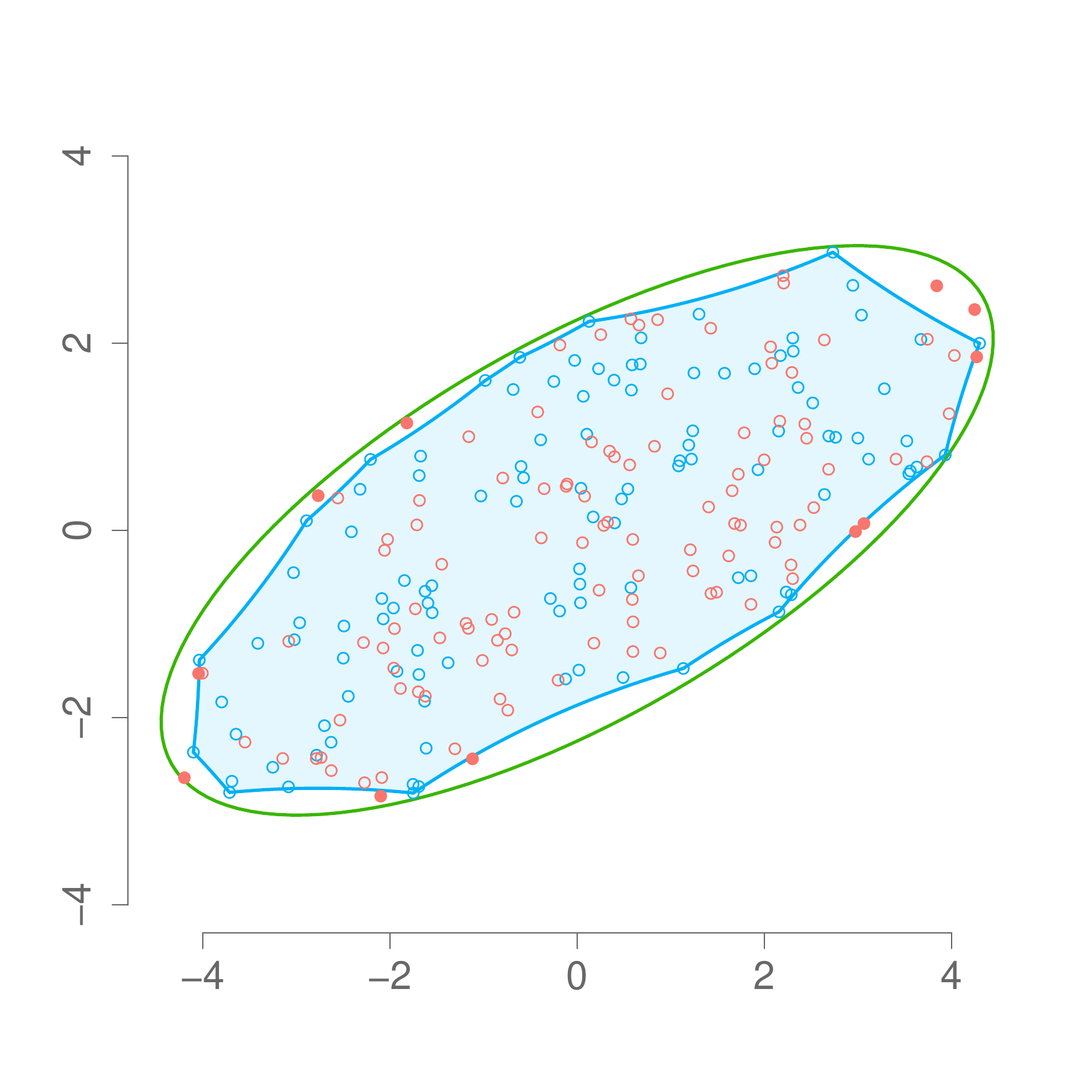}
\includegraphics[width=0.35\textwidth]{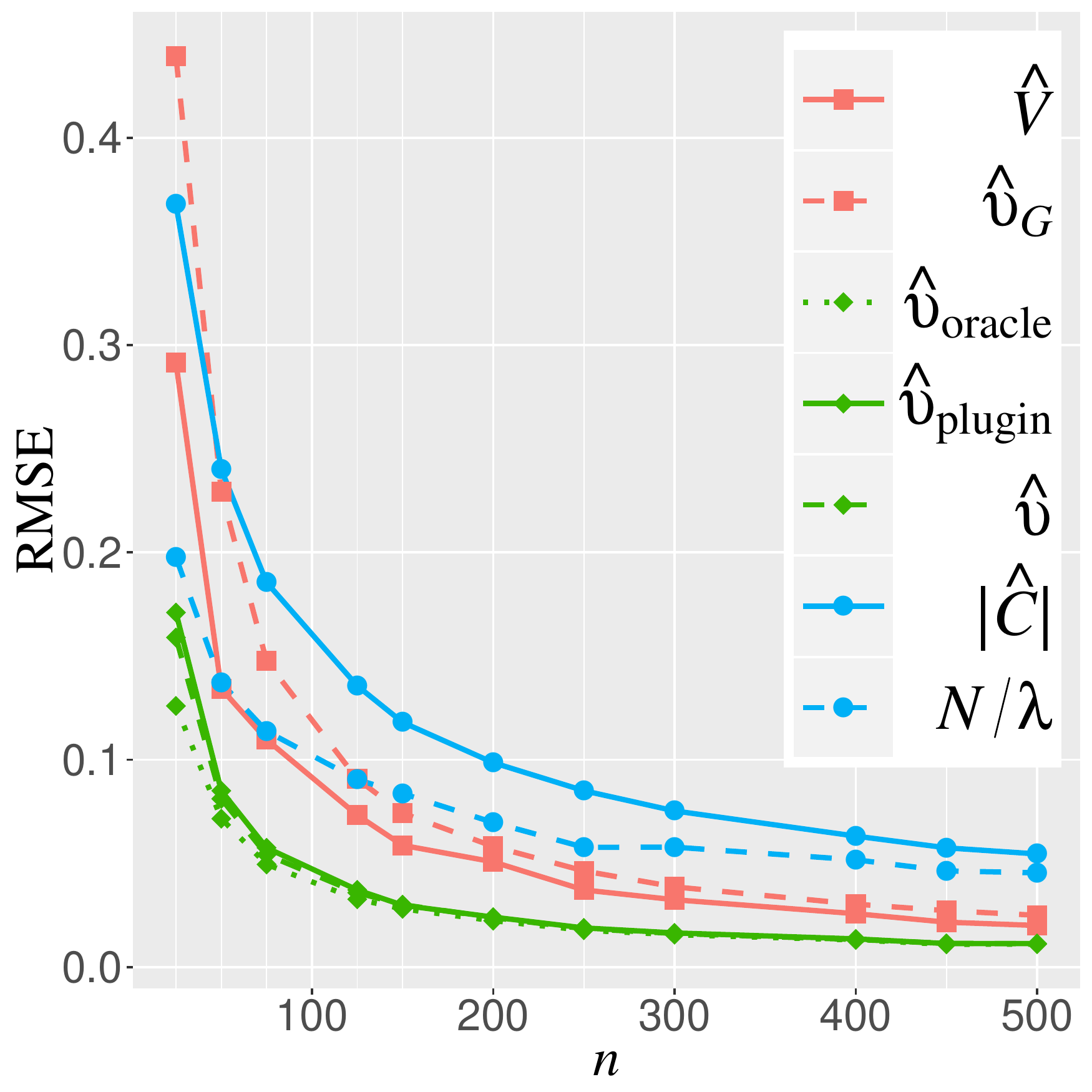}
\caption{Left: A random sample is generated on the ellipse $S$ in green and splitted into $\cX'_n$ (blue points) and $\cX''_n$ (open and solid red points). The solid red points are the observations of $\cX''_n$ that fall outside $C_\alpha(\cX'_n)$, represented in blue for $\alpha=10$. Right: The RMSE (normalized by $\mu(S)$) of different estimators of $\mu(S)$, based on $B = 500$ Monte Carlo iterations in each case.}
\label{fig:comp}
\end{figure}

\section{Discussion}
\label{sec:discussion}

\medskip\noindent {\em Choice of parameters.}
The first estimator we propose, just like the plug-in estimator, depends on the choice of $\alpha > 0$.  We proved 
our result (\thmref{simple}) under the assumption that $\alpha \le r$, but in general $r$ is unknown.  Also, although the theory works for any $\alpha$ thus chosen, in practice, an optimal choice for $\alpha$ may depend on the sample size.  Under uniform sampling, in \cite{rodriguez2014fully}
is proposed a data-driven selector of $\alpha$, $\alpha_n$, such that, with probability one, satisfies $\alpha_n\leq r$ and $\alpha_n\to r$. 

\medskip\noindent {\em Extensions.}
We are confident that our proof arguments proceed with relatively minor modifications when $\partial S$ is piecewise smooth.  However, our working condition \eqref{rolling} --- which is equivalently expressed as a requirement on the reach of $\partial S$ --- is simple and compact, and the resulting analysis already contains all the intricacies.
A more substantial extension is to the setting of an unknown sampling distribution.  This setting is considered in \citep{gayraud1997estimation}, where the sampling density is estimated by a standard kernel procedure, and that estimate is incorporated in the estimator of the volume.  Although the methodology and theory developed in that paper do not apply directly, an adaptation to our setting (namely, to sets satisfying \eqref{rolling}) seems viable.

\medskip\noindent {\em Perimeter.}
A parallel line of research tackles the problem of estimating the perimeter of $S$.
In fact, this problem is also considered by \cite{MR0169139} and \cite{MR1641826}, still in the context of a convex support set in dimension $d=2$.
More recently, in the same setting as ours here, but restricted to dimension $d=2$, \cite{cuefrapat12} study the perimeter of the sample $\alpha$-convex hull, while \cite{arias2015estimating} study the perimeter of the sample $\alpha$-shape.
Parallel to the work of \cite{MR1226450}, and working with a similar model, we find the work of \cite{kim2000estimation}.
We also mention a series of papers that consider the closely related problem of estimating the Minkowski content of the boundary of $S$, still under a similar model, making various regularity assumptions on $S$ \citep{CueFraRod07,patrod09,PatRod08,JimYuk11}.
It would be interesting to obtain similar results for the problem of estimating the perimeter of $S$ under our setting or some of these other settings.

\subsection*{Acknowledgments}
This work was partially supported by the US National Science Foundation (DMS 1513465) and by the Spanish Ministry of Economy and Competitiveness and ERDF funds (MTM2013-41383-P).

\appendix
\section{Appendix: additional proofs}
\subsection{Rolling a ball inside a convex set}
\begin{Def}
For a set $S$ and $\alpha > 0$, let $G_\alpha(S)$ denote the set of $x \in S$ with the property that there is an open ball $B$ of radius $\alpha$ such that $x \in B \subset S$.
\end{Def}

\begin{lem} \label{lem:convex}
If $S \subset \bbR^d$ is convex, then for any $\alpha > 0$, $G_\alpha(S)$ is either empty or convex.
\end{lem}

\begin{proof}
Write $G$ for $G_\alpha(S)$.  By definition (and the axiom of choice), for any $x \in S$ we may choose an open ball of radius $\alpha$, denoted $B_x$, such that $x \in B_x \subset S$.
Suppose $S$ contains a ball of radius $\alpha$, for otherwise $G$ is empty and there is nothing else to prove.  Take $x,y \in G$ and let $C$ denote the convex hull of $B_x \cup B_y$.  On the one hand, $C \subset S$, because $B_x \cup B_y \subset S$ and $S$ is convex.  On the other hand, for all $z \in C$, there is a ball $B$ of radius $\alpha$ such that $z \in B \subset C$.  This is obvious from the fact that $C$ is the union of the cylinder with center rod $[xy]$ and radius $\alpha$ and the two half balls defined by $B_x$ and $B_y$ on each end, which can be expressed as
\beq
C = \bigcup_{t \in [x,y]} B(t, \alpha).
\eeq
Hence, $C \subset G$, and in particular, $[xy] \subset G$ since $[xy] \subset C$.  This being true for all $x,y \in G$, we conclude that $G$ is indeed convex.
\end{proof}

\subsection{Proof of \thmref{cr_d}}
The arguments are only sketched and more details can be found in \citep{bea-tesis}.

Before proving \thmref{cr_d}, we need to introduce some notation.
The distance between $x \in \bbR^d$ and $S \subset \bbR^d$ is defined as $\dist(x, S) = \inf_{s \in S} \|x -s\|$.
Given a unit vector $u$ and an angle $\theta\in [0,\pi/2]$, consider the infinite cone with apex $x$, axis $u$ and aperture $2\theta$ defined by
\[C_{u}^\theta(x)=\{z\in\mathbb{R}^d, z\neq x:\left\langle {z-x},u\right\rangle\geq{\|z-x\|}\cos\theta\}.\]
Note that the notation is slightly different from the one we used in \secref{lower}. For $h>0$, consider the \textit{finite} cone obtained by intersecting an infinite cone with a ball of radius $h$ centered at its apex,
$C_{u,h}^\theta(x)=C_{u}^\theta(x)\cap B(x,h)$.
For $x\in\mathbb{R}^d$ and $r>0$, let $\cE_{x,r}=\{{B}(y,r): y\in {B}(x,r)\}$\glossary{$\cE_{x,r}$&$\{{B}(y,r): y\in {B}(x,r)\}$}.

\begin{Def} \label{def:inev2}
The family of sets $\cU$ is said to be unavoidable for another family of sets $\cE$ if, for all $E\in\cE$, there exists $U\in\cU$ such that $U \subset E$.
\end{Def}

\begin{lem}\label{lem:inev.Rd1}
Under the conditions of \thmref{cr_d} (and recalling the definition of $\delta$ there), for any $x\in S$ such that $\dist(x,\partial S)> \alpha_n/2$, there exists a finite family $\cU_{x,\alpha_n}$ with at most $m_1$ elements, unavoidable for $\cE_{x,\alpha_n}$ and that satisfies
\[P_X(U)\geq L_1\alpha_n^d, \quad \forall U\in \cU_{x,\alpha_n},\]
where $m_1 \ge 1$ depends only on $d$ and $L_1>0$ only on $(d,\delta)$.
\end{lem}

\begin{proof}
The case $d=1$ is handled separately. For $x\in \mathbb{R}$ under the stated conditions let us consider the unavoidable family
$\cU_{x,\alpha_n}=\{[x-\alpha_n/2,x],[x,x+\alpha_n/2]\}$. The result holds for $L_1=\delta/2$.  For the case $d=2$, see Proposition~1 of \citep{patrod13}. Let us then assume that $d\geq 3$ and fix $\theta=\pi/6$.
There exists a finite family $\W$, with $m_1$ unit vectors ($m_1$ dependent only on $d$), such that we can cover $B(x,\alpha_n)$ by the cones $C_{u,{\alpha_n}}^\theta(x)$, with $u\in\W$.
The family
$\cU_{x,\alpha_n}=\{C_{u,\alpha_n/2}^\theta(x),\ u\in\W\}$
is unavoidable for $\cE_{x,\alpha_n}$.
For each $u\in\W$,
\beq\label{buf66}
P_X\left(C_{u,\alpha_n/2}^\theta(x)\right)\geq\delta\mu\left(C_{u,\alpha_n/2}^\theta(x)\cap S\right)=\delta\mu\left(C_{u,\alpha_n/2}^\theta(x)\right)
\eeq
where the equality comes from the fact that $\dist(x,\partial S)> \alpha_n/2$.
Finally,
\beq\label{buf666}
\mu\left(C_{u,\alpha_n/2}^\theta(x)\right)\geq\mu(B(x,\alpha_n/2))/m_1=\omega_d\left(\alpha_n/2\right)^d/m_1,
\eeq
where $\omega_d$ denotes the Lebesgue measure of the unit ball in $\mathbb{R}^d$.  The proof is complete with $L_1=\delta\omega_d/(2^dm_1)$.
\end{proof}

\begin{lem}\label{lem:inev.Rd2}
Under the conditions of \thmref{cr_d} (and recalling the definition of $r$ and $\delta$ there), for any $x\in S$ such that $\dist(x,\partial S)\leq \alpha_n/2$, there exists a finite family $\cU_{x,\alpha_n}$ with at most $m_2$ elements, unavoidable for $\cE_{x,\alpha_n}$ and that satisfies
\[P_X(U)\geq L_2\alpha_n^{(d-1)/2}\dist(x,\partial S)^{(d+1)/2}, \quad \forall U\in \cU_{x,\alpha_n},\]
where $m_2 \ge 1$ depends only on $d$ and $L_2>0$ only on $(d, r, \delta)$.
\end{lem}

\begin{proof}
Let $x\in S$ such that $\dist(x,\partial S)\leq \alpha_n/2$. We denote $\rho=\dist(x,\partial S)$. For $d=1$ consider the unavoidable family
$\cU_{x,\alpha_n}=\{[x-\rho,x],[x,x+\rho]\}$ and the result holds for $L_2=\delta$.
For the case $d=2$, see Proposition~2 of \citep{patrod13}.
Let us assume that $d\geq 3$. Using the rolling condition, let $P_{\Gamma}x$ be the metric projection of $x$ onto $\Gamma:=\partial S$ and $\eta$ the outward pointing unit normal vector at $P_{\Gamma}x$. By the fact that $S^\comp$ satisfies the $r$-rolling ball condition, we have that $B(P_{\Gamma}x-r\eta,r)\subset S$.
Thus, given $\cU_{x,\alpha_n}$ an unavoidable family of sets for $\cE_{x,\alpha_n}$, we have that
\beq
P_X(U)\geq \delta \mu(U\cap S)\geq \delta\mu(U\cap B(P_{\Gamma}x-r\eta,r))
\eeq
for all $U\in \cU_{x,\alpha_n}$. We can assume, without loss of generality, that $x$ is the origin and $\eta=-e_d$, where $e_d$ denotes the $d$-th canonical basis vector.
Then, the problem reduces to defining a suitable family of sets $\cU_{0,\alpha_n}$ unavoidable for  $\cE_{0,\alpha_n}$ and giving a lower bound for $\mu(U\cap B((r-\rho)e_d,r))$ independent of $U\in \cU_{0,\alpha_n}$. 
We partition $B(0,\alpha_n)$ into the following two sets

\beq
\G_{\alpha_n}=\left\{y\in B(0,\alpha_n):\ \left\langle y, e_d\right\rangle\geq -\left\|y\right\|/2\right\}, \quad
\F_{\alpha_n}=\left\{y\in B(0,\alpha_n):\ \left\langle y, e_d\right\rangle< -\left\|y\right\|/2\right\}.
\eeq
In order to simplify the notation, we write $C_{u}^\theta$ and $C_{u,h}^\theta$ to refer to $C_{u}^\theta(x)$ and $C_{u,h}^\theta(x)$ for $x=0$.

First, let us consider $\G_{\alpha_n}$. Fix $\theta=\pi/6$ and $\gamma\in (0,\pi/6)$, say $\gamma = \pi/7$. There exists a finite family $\W^\G$, with $m^\G$ unit vectors (depending only on $d$), with the property that for all $y\in\G_{\alpha_n}$ there exists  $u\in\W^\G$ such that $C_{u,\alpha_n}^\theta\subset B(y,\alpha_n)$ and $\left\langle u,e_d\right\rangle\geq -\sin\gamma$.
Let $H_0=\{x\in \mathbb{R}^d : \<x, e_d\>  \geq 0\}$.  There is an absolute angle $\tilde{\theta} > 0$ with the property that, for each unit vector $u$ with $\left\langle u,e_d\right\rangle\geq-\sin\gamma$ there exists a unit vector $\tilde{u}$ such that $C_{\tilde{u}}^{\tilde{\theta}}\subset C_{u}^\theta\cap H_0$. Now, for $\psi=\sqrt{\rho(2r-\rho)}$ and for each $u\in\W^\G$,
\beq
C_{u,\alpha_n}^\theta\cap B((r-\rho)e_d,r)\supset C_{u,\alpha_n}^\theta\cap H_0\cap B(0,\psi)\supset C_{\tilde{u},\alpha_n}^{\tilde{\theta}}\cap B(0,\psi)=C_{\tilde{u},\tau_n}^{\tilde{\theta}},
\eeq
where $\tau_n:=\min(\psi,\alpha_n)$. The ball $B(0,\tau_n)$ can be covered by a finite number $m$ (depending only on $d$) of cones $C_{\tilde{u},\tau_n}^{\tilde{\theta}}$, with varying $\tilde u$.
Using that $\alpha_n\leq r$ and $\rho\leq \alpha_n/2$, we have that
\beq
\mu\left(B(0,\tau_n)\right)\geq w_d\alpha_n^{(d-1)/2}\rho^{(d+1)/2}
\eeq
and, therefore,
\beq
\mu(C_{u,\alpha_n}^\theta\cap B((r-\rho)e_d,r))\geq L^\G\alpha_n^{(d-1)/2}\rho^{(d+1)/2},
\eeq
where $L^\G:=w_d/m>0$ only depends on $d$.

Now, let us consider $\F_{\alpha_n}$. First, we define the set
\beq
\C_\ddag=\{x\in\mathbb{R}^d:\ -h_1\leq \left\langle x,e_d\right\rangle\leq 0\}\cap B(-\alpha_n e_d,\alpha_n),
\eeq
where $h_1:={\rho(2r-\rho)}/(2(r+\alpha_n-\rho))$; see \figref{QuCh} (left). It can be proved that $\C_\ddag\subset  B((r-\rho)e_d,r)$ and
\beq\label{cmeas}
\mu(\C_\ddag)\geq\frac{\omega_{d-1}}{(d+1)2^{(d-1)/2}}\alpha_n^{(d-1)/2}\rho^{(d+1)/2}.
\eeq
For $\theta\in [0,\pi/2]$ and a unit vector $u\in\mathbb{R}^{d-1}$, let
$Q_u^\theta=\left\{x\in\mathbb{R}^d:\ x_{-d} \in C_{u}^\theta\right\}$, where $x_{-d}$ denotes the  vector $x$ without the last component; see \figref{QuCh} (right).
\begin{figure}[!ht]
\centering
\includegraphics[width=0.3\textwidth]{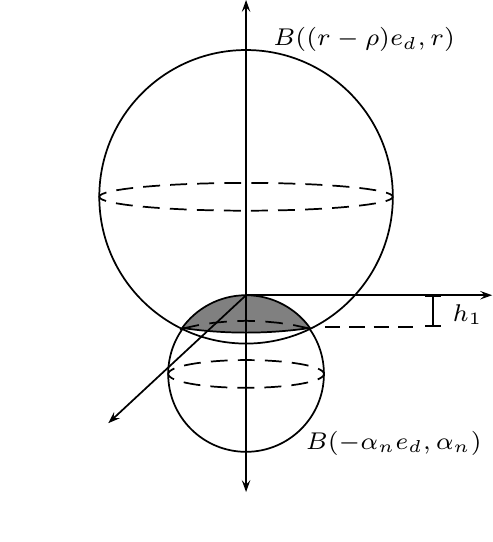}
\includegraphics[width=0.3\textwidth]{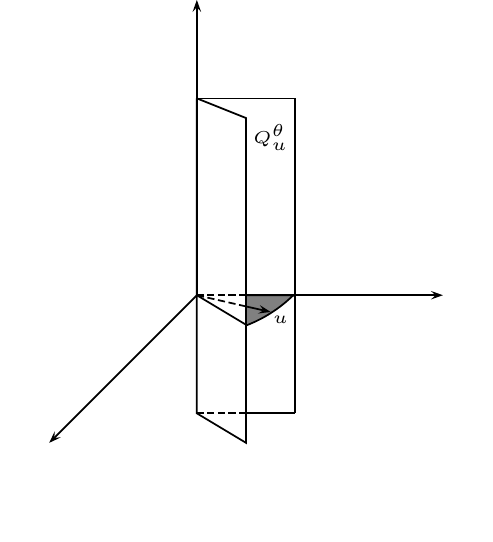}
\caption{Left, in gray  set $\C_\ddag$ in $\mathbb{R}^3$. Right, example of set $Q_u^\theta$ in  $\mathbb{R}^3$.}
\label{fig:QuCh}
\end{figure}
We will prove that we can define an unavoidable family with sets of the form $Q_u^\theta\cap\C_\ddag$, all with the same Lebesgue measure; see \figref{QuCh2}.
\begin{figure}[!ht]
\centering
\includegraphics[width=0.25\textwidth]{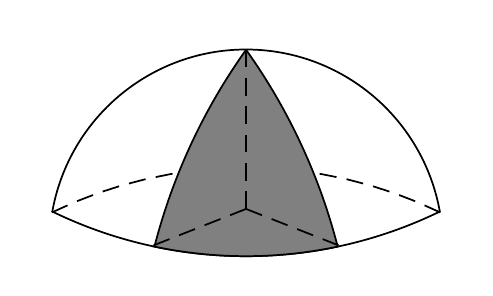}
\caption{Example of set $Q_u^\theta\cap\C_\ddag$ in  $\mathbb{R}^3$.}
\label{fig:QuCh2}
\end{figure}
Fix $\theta=\pi/6$. There exists a finite family $\W^\F$, with $m^\F$ unit vectors in $\mathbb{R}^{d-1}$ (depending only on $d$), with the property that for all $y\in\F_{\alpha_n}$ there exists  $u\in\W^\F$ such that $Q_u^\theta\cap\C_\ddag\subset B(y,\alpha_n)$.
Finally, we use that $\C_\ddag\subset  B((r-\rho)e_d,r)$, the fact that $\C_\ddag$ can be covered by the sets $Q_u^\theta\cap\C_\ddag$ with $u\in\W^\F$, and \eqref{cmeas}, to obtain the following sequence of inequalities
\beq
\mu\left(Q_u^\theta\cap\C_\ddag\cap B((r-\rho)e_d,r)\right)=\mu(Q_u^\theta\cap\C_\ddag)\geq\mu(\C_\ddag)/m^\F\geq  L^{\F} \alpha_n^{(d-1)/2}\rho^{(d+1)/2},
\eeq
where $L^\F:=\frac{\omega_{d-1}}{m^F(d+1)2^{(d-1)/2}}$ depends only on $d$.

We finish by defining the family
\beq
\cU_{0,\alpha_n}= \{C_{u,\alpha_n}^\theta: u\in\W^{\G}\}\cup\{Q_u^\theta\cap\C_\ddag: u\in\W^{\F}\}.
\eeq
This completes the proof of the lemma, with $m_2:=m^\G+m^\F$ depending only on $d$ and $L_2 :=\delta\min(L^\G,L^\F)$ only on $(d, \delta)$. \end{proof}

\begin{proof}[Proof of \thmref{cr_d}]
Let $S_n=C_{\alpha_n}(\cX_n)$. With probability one, $\mathcal{X}_n\subset S$, which implies $S_n\subset S$ and
\beq\label{esp}
\E\big[\mu(S_n \symd S)\big]=\mathbb{E}\big[\mu(S\setminus S_n)\big]
=\int_S{P(\exists y\in {B}(x,\alpha_n): {B}(y,\alpha_n)\cap\mathcal{X}_n=\emptyset)\mu({\rm d}x)}.
\eeq
For each $x\in S$ we choose a finite family {\mbox{$\cU_{x,\alpha_n}$}} unavoidable for $\cE_{x,\alpha_n}$.  Then, as a consequence of \defref{inev2}, we have
\begin{equation}
\label{esp2}
\begin{split}
P(\exists y\in {B}(x,\alpha_n): {B}(y,\alpha_n)\cap\mathcal{X}_n=\emptyset)
\le \sum_{U\in\cU_{x,\alpha_n}}(1-P_X(U))^n
\le \sum_{U\in\cU_{x,\alpha_n}}\exp\left(-nP_X(U)\right).
\end{split}
\end{equation}
We partition $S$ into two subsets
\beq
S_1=\left\{x\in S:\ \ \dist(x,\partial S)>\alpha_n/2\right\}, \quad
S_2=\left\{x\in S:\ \ \dist(x,\partial S)\leq  \alpha_n/2\right\}.
\eeq
For those $x\in S_1$, choose a family as in \lemref{inev.Rd1}, to get
\beq\label{O1d}
\int_{S_1}{\sum_{U\in\cU_{x,\alpha_n}}\exp(-nP_X(U))\mu({\rm d}x)}
\leq\int_{S_1}{m_1\exp(-nL_1\alpha_n^d)\mu({\rm d}x)}
=O\left(\textnormal{e}^{-L_1n\alpha_n^d}\right),
\eeq
where $m_1$ and $L_1$ are defined in \lemref{inev.Rd1}.
For those $x\in S_2$, choose a family as in \lemref{inev.Rd2}, and follow the same arguments as in \citep{patrod13}, to get
\beq\label{O2d}
\begin{split}
\int_{S_2}{\sum_{U\in\cU_{x,\alpha_n}}\exp(-nP_X(U))\mu({\rm d}x)}
&\leq\int_{S_2}{m_2\exp\left(-L_2n\alpha_n^{(d-1)/2}\dist(x,\partial S)^{(d+1)/2}\right)\mu({\rm d}x)}\\
&=O\left(\alpha_n^{-(d-1)/(d+1)}n^{-2/(d+1)}\right).
\end{split}
\eeq
It follows from \eqref{O1d} and \eqref{O2d}, and all the derivations that precede these, that
\beq\label{larga}
\E\big[\mu(C_{\alpha_n}(\cX_n) \symd S)\big]=O\left(\textnormal{e}^{-L_1n\alpha_n^d}+\alpha_n^{-(d-1)/(d+1)}n^{-2/(d+1)}\right).
\eeq
Since $\alpha_n$ is bounded by $r$ and $n \alpha_n^d/\log n \to \infty$, we have $\textnormal{e}^{-L_1n\alpha_n^d}=o(\alpha_n^{-(d-1)/(d+1)}n^{-2/(d+1)})$, and this completes the proof.
\end{proof}

\small
\bibliographystyle{chicago}
\bibliography{volume-minimax}

\end{document}